\documentclass[a4j]{article}
\usepackage{amsthm}
\usepackage{amssymb}
\usepackage{amsmath}
\usepackage{mathrsfs}
\usepackage{enumerate}
\usepackage{latexsym,color,ulem}

\newtheorem{defi}{Definition}
\newtheorem{lem}{Lemma}

\newtheorem{thm}{Theorem}

\newtheorem{rem}{Remark}

\newtheorem{assump}{Assumption}

\newcommand{\argmin}{\mathop{\rm arg~min}\limits}

\title{Asymptotically efficient estimators for self-similar stationary Gaussian noises under high frequency observations}
\author{Masaaki Fukasawa and Tetsuya Takabatake\footnote{Graduate School of Engineering Science,
		Osaka University, 1-3 Machikaneyama, Toyonaka, Osaka, Japan. 
		Email: takabatake@sigmath.es.osaka-u.ac.jp}}
\date{}

\begin{document}
	\maketitle　
	\begin{abstract}
		This paper proposes feasible asymptotically efficient estimators for a certain class of 
		Gaussian noises with self-similar and stationary properties, 
		which includes the fractional Gaussian noise, under high frequency observations. 
		In this setting, the optimal rate of estimation depends on whether either the Hurst or diffusion parameters is known or not.
		This is due to the singularity of the asymptotic Fisher information matrix 
		for simultaneous estimation of the above two parameters. 
		One of our key ideas is to extend the Whittle estimation method 
		to the situation of high frequency observations.
		We show that our estimators are asymptotically efficient in Fisher's sense. 
	\end{abstract}
	
	\section{Introduction}
	Self-similar and Gaussian properties of noises in time series data are observed in many fields, 
	for example, hydrology, turbulence, molecular biology and financial economics. 
	Fractional Brownian motion, which is introduced by Kolmogorov (1940) and further developed by Manderblot and Van Ness (1968), 
	is the most fundamental continuous-time model to represent these phenomena. 
	Until now, statistical inference problems of discretely observed fractional Brownian motion have been studied 
	under the large sample asymptotics (cf. Fox and Taqqu 1986, Dahlhaus 1989, Liberman et al. 2009, 2012, Cohen et al. 2013) 
	or the high frequency asymptotics (cf. Coeurjolly 2001, Brouste and Iacus 2013, Kawai 2013, Brouste and Fukasawa 2016). 
	Some of them also discuss the optimality of their estimators or local asymptotic normality (LAN) property under those asymptotics. 
	To the best of our knowledge, however, there is no estimator  available so far which is feasible and asymptotic optimal under the high frequency asymptotics,
	despite that it has become important because of increasing availability of high frequency data thanks to recent developments of information technology.
	
	Now, we review the results of Kawai (2013) and Brouste and Fukasawa (2016),  
	where they studied the LAN property of the fractional Gaussian noises under high frequency observations. 
	The parameter to be estimated is $(H, \sigma)\in(0, 1]\times(0, \infty)$, 
	where $H$ and $\sigma$ are the so-called Hurst parameter and  diffusion parameter respectively.
	Kawai (2013) obtained a LAN property in a weak sense, where
	the rate matrix $\bar{\phi}_N(H, \sigma)$ and the asymptotic Fisher information matrix $\mathcal{I}(H, \sigma)$ are given by 
	\begin{align*}
		\bar{\phi}_N(H, \sigma)=
		\begin{bmatrix}
			\frac{1}{\sqrt{N}|\log\delta|}& 0 \\
			0 & \frac{1}{\sqrt{N}} \\
		\end{bmatrix}
		\hspace{0.2cm},\hspace{0.2cm}
		\mathcal{I}(H, \sigma)=
		\begin{bmatrix}
			2 & -\frac{2}{\sigma} \\
			-\frac{2}{\sigma} & \frac{2}{\sigma^2} \\
		\end{bmatrix}.
	\end{align*}
	Here, $N$ is the sample size and $\delta$ is the length of sampling intervals.
	The LAN property is only in a weak sense because $\mathcal{I}(H, \sigma)$ is a singular matrix. 
	As a result, the asymptotic lower bounds of risk are derived only in the case that 
	either the Hurst parameter $H$ or the diffusion parameter $\sigma$ is known.
	In Brouste and Fukasawa (2016), it was shown that the LAN property in the usual sense actually holds even if both of the parameters are unknown.
	One of their key ideas is using a certain class of non-diagonal rate matrices. 
	The LAN property enables them to determine the optimal rate of estimation for each of $H$ and $\sigma$.
	As we see later in this paper, the lower bounds of variance can be calculated more explicitly; for any estimators $\hat{H}_N$ and $\hat{\sigma}_N$,
	\begin{align*}
		\liminf_{\eta\to 0}\liminf_{\delta\to 0}\sup_{|(H, \sigma)-(H_0, \sigma_0)|<\eta}NE_{H, \sigma}[(\hat{H}_N-H)^2]\geq& \mathcal{J}(H_0)^{-1},\\
		\liminf_{\eta\to 0}\liminf_{\delta\to 0}\sup_{|(H, \sigma)-(H_0, \sigma_0)|<\eta}\frac{N}{\sigma^2\left(\log\delta\right)^2}E_{H, \sigma}[(\hat{\sigma}_N-\sigma)^2]\geq&\mathcal{J}(H_0)^{-1},
	\end{align*}
	for any $(H_0, \sigma_0)$, where
	\begin{align*}
		\mathcal{J}(H)=\frac{1}{4\pi}\int_{-\pi}^\pi\left[\frac{\partial}{\partial H}\log{g_H(\lambda)}\right]^2\, d\lambda.
	\end{align*}
	The definition of the function $g_H$ will be given later. 
	It is noteworthy that the optimal rates of estimation are slower than in the case where the other parameter is known.
	
	The  LAN property  implies  the asymptotic efficiency of the maximum likelihood estimator (MLE) in general. 
	The MLE is unfortunately not feasible for the fractional Gaussian noises because the computation of the inverse of the covariance matrices is very heavy. 
	Our main contribution in this context is to construct feasible asymptotically efficient estimators.
	We deal with the three cases: 
	both the Hurst parameter $H$ and the diffusion parameter $\sigma$ are unknown, 
	only the Hurst parameter $H$ is known, only the diffusion parameter $\sigma$ is known. 
	Actually, we work under a more general model of Gaussian noises with self-similar and stationary properties, 
	which generalizes the fractional Gaussian noises. 
	
	This paper is constructed as follows. 
	In Section \ref{Model}, we introduce the model. 
	In Section \ref{Examples}, we show several examples in our framework. 
	In Section \ref{Main}, we present and prove our main results. 
	In Appendix, we give several lemmas and extensions of the results in Kawai (2013), Brouste and Fukasawa (2016). 
	
	\section{Model Assumption}\label{Model}
	In this section, we introduce several assumptions to be satisfied by our model, 
	which are inspired from several previous works, for example, Fox and Taqqu (1986), Dahlhaus (1989, 2006), 
	Lieberman et al. (2009) and Cohen et al. (2013). 
	\begin{assump}\label{model.assump.} \rm 
		Let $\delta>0$ be the length of sampling intervals and $N=N(\delta)$ be the size of data.
		We assume $N\to\infty$, $\delta\to 0$ and $\inf_{\delta>0}N\delta>0$.
		Let $\Theta$, $\Sigma$ be compact subsets of $\mathbb{R}^{p-1}\times(0,1]$, $(0,\infty)$, 
		and true values $\theta_0=(\psi_0, H_0)=(\psi_0^{(1)}, \cdots, \psi_0^{(p-1)}, H_0)$, $\sigma_0$ be interior points of $\Theta$, $\Sigma$ respectively.
		We assume that $\{X_j^\delta\}_{j=1,\dots, N}$ is a stationary Gaussian sequence with mean $0$
		and spectral density $f^\delta(\theta, \sigma, \lambda)$, 
		$\lambda\in [-\pi,\pi]$ satisfying the following conditions $(A.\ref{A.decomposition})-(A.\ref{spd.asyrelation})$. 
		\begin{enumerate}[$({A}.1)$]
			\setcounter{enumi}{-1}
			\item{$f^\delta(\theta, \sigma, \lambda)$ is decomposed as $\sigma^2{\delta}^{2H}f(\theta, \lambda)$, 
				where $f\in\mathcal{C}^{3, 1}\left(\Theta\times[-\pi, \pi]/\{0\}\right)$.}\label{A.decomposition}
		\end{enumerate}
		In the following, we denote $f^\delta(\theta, \sigma, \lambda)$, $f(\theta, \lambda)$ as $f_{\theta, \sigma}^{\delta}(\lambda)$, $f_{\theta}(\lambda)$ respectively. 
		\begin{enumerate}[$({A}.1)$]
			\item{
				If $(\theta_1, \sigma_1)$ and $(\theta_2, \sigma_2)$ are distinct elements of $\Theta$, 
				a set $\{\lambda\in [-\pi,\pi]:\sigma_1 f_{\theta_1}(\lambda)\neq \sigma_2 f_{\theta_2}(\lambda)\}$ has a positive Lebesgue measure. }\label{A.identify}
			\item{There exists a continuous function $\alpha:\Theta\times\Sigma\to(-1,1)$ such that for any $\eta>0$, 
				there exist positive constants $c_{1, \eta}, c_{2, \eta}$, which only depend on $\eta$, 
				such that the following conditions hold for every $(\theta, x)\in\Theta\times[-\pi, \pi]\backslash\{0\}$.
				\begin{enumerate}[$(a)$]
					\item $c_{1, \eta}|\lambda|^{-\alpha(\theta, \sigma)+\eta}\leq f_{\theta}(\lambda)\leq c_{2, \eta}|\lambda|^{-\alpha(\theta, \sigma)-\eta}$.
					\item For any $l\in\{1, 2, 3\}$ and $k\in\{1, \cdots, p\}^l$,
					\begin{align*}
						\left|\frac{\partial^l}{\partial\theta_{k_1}\cdots\partial\theta_{k_l}}f_\theta(\lambda)\right|&\leq c_{2, \eta}|\lambda|^{-\alpha(\theta, \sigma)-\eta}, \\
						\left|\frac{\partial^{l+1}}{\partial\lambda \partial\theta_{k_1}\cdots\partial\theta_{k_l}}f_\theta(\lambda)\right|&\leq c_{2, \eta}|\lambda|^{-\alpha(\theta, \sigma)-1-\eta}.
					\end{align*}
				\end{enumerate}}\label{spd.asyrelation}
			\end{enumerate}
		\end{assump}
		
		\begin{rem}\rm
			Under the above assumptions, the spectral density $f_{\theta, \sigma}^\delta(\lambda)$ and its derivatives are integrable.  
			Moreover, we can exchange the differential and integral operators freely. For example, it holds that
			\begin{align*}
				\frac{\partial}{\partial\theta_j}\int_{-\pi}^\pi \log{\left(\sigma^2 f_\theta(\lambda)\right)}\, d\lambda 
				= \int_{-\pi}^\pi \frac{\partial}{\partial\theta_j}\log{\left(\sigma^2 f_\theta(\lambda)\right)}\, d\lambda,
			\end{align*}
			for $j=1, \cdots, p+1$. We denote $\theta_{p+1} = \sigma$ for notational simplicity. 
		\end{rem}
		
		\begin{rem}\rm
			The condition (A.\ref{A.decomposition}) is a major difference from the model assumptions in the previous works. 
			In the large sample situation, that is, $\delta$ is fixed and $N\to\infty$, we need not to consider this decomposition of the spectral density 
			because we can include the part $\sigma^2\delta^{2H}$ into the model parameters.  
			This is also mentioned in the seminal paper of Dahlhaus (1989), p.1752. 
			In the case of high frequency observations, however, we can not do this because the variable $\delta$ 
			that drives asymptotics is heavily related with the Hurst parameter $H$. 
			Moreover, the parameter $\sigma$ and the others play different roles under high frequency observations.  
			In fact, as we see later, the efficient rate for estimators of $\sigma$ is different from that for the others.
		\end{rem}
		
		\section{Examples}\label{Examples}
		In this section, we give several examples satisfying the assumptions introduced in the previous section. 
		Namely, we treat two models: the fractional Brownian motion and a special case of the fractional Langevin model.
		\subsection{Fractional Brownian Motion}
		A centered Gaussian process $B^H$ is called a fractional Brownian motion with the Hurst parameter $H$ 
		if its covariance structure is given by
		\begin{align*}
			E[B_t^HB_s^H]=\frac{1}{2}(|t|^{2H}+|s|^{2H}-|t-s|^{2H}), s, t\in\mathbb{R}.
		\end{align*}
		Such a process exists and is continuous for all $H\in(0, 1]$ from the Kolmogorov's extension and continuous theorems. 
		Moreover, the process is characterized by stationary increments and self-similar properties: 
		for any $\delta>0$ and $t, k\in\mathbb{R}$,
		\begin{align}\label{prop.fBm}
			B_{t+k}^H-B_t^H \sim B_k^H\hspace{0.2cm}\mbox{and}\hspace{0.2cm} B_{\delta t}^H&\sim \delta^H B_t^H\hspace{0.2cm}\mbox{in law}.
		\end{align}
		
		Let $\delta>0$ and $T=1$ for notational simplicity. 
		Assume we have the following observations:
		\begin{align*}
			\sigma B_{\delta}^H, \sigma B_{2\delta}^H, \cdots, \sigma B_{N\delta}^H,
		\end{align*}
		where $N$ is the sample size and $\delta$ is the length of sampling intervals. 
		Let us consider the high frequency asymptotics, that is, $N\to\infty$ as $\delta\to 0$. 
		Define a sequence $\{X_j^\delta\}_{j=1, \cdots, N}$ by
		\begin{align*}
			X_j^\delta=\sigma\Delta B_{j\delta}^H=\sigma B_{j\delta}^H-\sigma B_{(j-1)\delta}^H.
		\end{align*}
		Note that $\{X_j^\delta\}_{j=1, \cdots, N}$ is a stationary centered Gaussian sequence 
		from (\ref{prop.fBm}). Moreover, its spectral density is characterized by 
		\begin{align*}
			E[X_1^\delta X_k^\delta]=&\frac{\sigma^2\delta^{2H}}{2}(|k+1|^{2H}-2|k|^{2H}+|k-1|^{2H}) \\
			=&\frac{1}{2\pi}\int_{-\pi}^\pi e^{\sqrt{-1}k\lambda}f_{H, \sigma}^\delta(\lambda)\, d\lambda,
		\end{align*}
		where
		\begin{align*}
			f_{H, \sigma}^\delta(\lambda)=&\sigma^2\delta^{2H}f_{H}(\lambda)\\
			=&\sigma^2\delta^{2H}C_H2(1-\cos{\lambda})\sum_{j=-\infty}^\infty\frac{1}{(\lambda+2\pi j)^{1+2H}},
		\end{align*}
		with $C_H=(2\pi)^{-1}\Gamma(2H+1)\sin(\pi H)$. Then, we can easily check the stationary centered Gaussian sequence 
		$\{X_j^\delta\}_{j=1, \cdots, N}$ satisfy all conditions $(A.\ref{A.decomposition})-(A.\ref{spd.asyrelation})$. 
		(cf. Fox and Taqqu (1986) and Rosemarin (2008)) 
		
		\subsection{Fractional Langevin Model}
		In the context of molecular biology, the movement of particle in homogeneous medium is modeled by the Langevin equations. 
		One of the characteristics of these equations is that a mean square displacement of particle linearly grows in time. 
		The particle in inhomogeneous medium, however, does not behave in the same way. 
		Namely, the mean square displacement of particle in this situation grows as a power function in time. 
		This phenomena is the so-called anomalous diffusion. 
		Therefore, we attempt to model this by the following second order stochastic differential equations:
		\begin{align*}
			&dZ_t=Y_t\,dt \\
			&dY_t=-\nabla q(Z_t)-\gamma Y_t\, dt+\sigma dB_t^H,\gamma, \sigma>0,
		\end{align*}
		where $Z$, $Y$ represent the position and velocity of particle respectively, $q$ is the potential, 
		$\gamma, \sigma$ are the friction and diffusion coefficients respectively, and $B^H$ is the fractional Brownian motion with the Hurst parameter $H$. 
		Here, we assumed the mass of the particle is equal to 1 for notational simplicity. 
		Therefore, we call the above equations as a fractional Langevin model named after the Langevin model in the case of $H=1/2$. 
		
		Let $\delta>0$, $T=1$. Here, we consider a situation of a free particle with no friction term, that is, 
		we assume the potential $q$ is a constant function and the friction coefficient $\gamma=0$ under the fractional Langevin model. 
		Then, we consider the statistical inference problem for the Hurst parameter $H$ and the diffusion parameter $\sigma$ 
		from high frequency position data:
		\begin{align*}
			Z_{\delta}, Z_{2\delta}, \cdots, Z_{N\delta}, 
		\end{align*}
		where $N$ is the sample size, $\delta$ is the length of sampling intervals 
		which satisfy $N\to\infty$ as $\delta\to 0$. 
		In this situation, we have no velocity data so that we consider to substitute numerical derivatives of the position data $\{Z_j\}_{j=1, \cdots, N}$ for the velocity data.  
		Namely, we set the proxy data $\{Y_j^\delta\}_{j=1, \cdots, N}$ for the velocity data of particle as follows:
		\begin{align*}
			Y_j^\delta=\frac{1}{\delta}(Z_{j\delta}-Z_{(j-1)\delta})=\frac{1}{\delta}\int_{(j-1)\delta}^{j\delta}Y_t\, dt.
		\end{align*}
		If we set a sequence $\{X_j^\delta\}_{j=2, \cdots, N}$ as differences of $\{Y_j^\delta\}_{j=1, \cdots, N}$:
		\begin{align*}
			X_j^\delta=Y_j^\delta-Y_{j-1}^\delta=\frac{\sigma}{\delta}\int_{(j-1)\delta}^{j\delta}(B_t^H-B_{t-\delta}^H)\, dt,
		\end{align*}
		then we can easily show that $\{X_j^\delta\}_{j=2, \cdots, N}$ is a Gaussian sequence and its spectral density is characterized by 
		\begin{align*}
			E[X_1^\delta X_k^\delta]=&\frac{\sigma^2\delta^{2H}}{2} (|k+2|^{2H+2}-4|k+1|^{2H+2} \\
			& \hspace{1cm} +6|k|^{2H+2}-4{|k-1|}^{2H+2}+{|k-2|}^{2H+2}) \\
			=&\frac{1}{2\pi}\int_{-\pi}^\pi e^{\sqrt{-1}k\lambda}f_{H, \sigma}^\delta(\lambda)\, d\lambda,
		\end{align*}
		where 
		\begin{align*}
			f_{H, \sigma}^\delta(\lambda)=&\sigma^2\delta^{2H}f_{H}(\lambda)\\
			=&\sigma^2\delta^{2H}C_H\{2(1-\cos{\lambda})\}^2\sum_{j=-\infty}^\infty\frac{1}{(\lambda+2\pi j)^{3+2H}},
		\end{align*}
		with $C_H=(2\pi)^{-1}\Gamma(2H+1)\sin(\pi H)$. In particular, $\{X_j^\delta\}_{j=2, \cdots, N}$ is a stationary sequence 
		and satisfy the condition (A.\ref{A.decomposition}). 
		Moreover, we can also prove the behavior of the spectral density and its derivatives at the origin satisfy 
		the conditions $(A.\ref{A.identify})-(A.\ref{spd.asyrelation})$ in the similar ways of the previous one. 
		As the above, this example is also included in our model framework. 
		
		\section{Construction of Asymptotically Efficient Estimators}\label{Main}
		The Whittle estimation method is very useful to estimate the Hurst parameter for several stationary Gaussian time series 
		in various aspects, for example, the Whittle estimator enjoys asymptotic efficiency as well as the MLE 
		and can be computed easier and faster than it because we compute an approximated log-likelihood 
		instead of the exact one which include the inverse of covariance matrix. 
		Therefore, we attempt to prove the Whittle estimation method also works well under the high frequency observations. 
		In the following, we consider statistical inference problems for our model in three cases: 
		all parameters $(\theta, \sigma)$ are unknown, only the Hurst parameter $H$ is known, 
		only the diffusion parameter $\sigma$ is known.
		In each setting, we construct an feasible asymptotically efficient estimator by applying the Whittle estimation method. 
	
		\subsection{Notation}
		In this subsection, we prepare notation. Let 
		\begin{align*}
			(\mathcal{X}_n, \mathcal{A}_n, \{P_{\theta, \sigma}^{(\delta)}; (\theta, \sigma)\in\Theta\times\Sigma\})
		\end{align*}
		be a statistical experiments for our model and set
		\begin{align*}
			a_q(\theta)=&\left(\frac{1}{2\pi}\int_{-\pi}^\pi\frac{\partial}{\partial\theta_1}\log{f_\theta(\lambda)}\, d\lambda
			, \cdots, \frac{1}{2\pi}\int_{-\pi}^\pi\frac{\partial}{\partial\theta_q}\log{f_\theta(\lambda)}\, d\lambda\right), \\
			\mathcal{F}_q(\theta)=&\left(
			\frac{1}{4\pi}\int_{-\pi}^\pi\frac{\partial}{\partial\theta_j}\log{f_\theta(\lambda)}\frac{\partial}{\partial\theta_k}\log{f_\theta(\lambda)}\, d\lambda
			\right)_{j, k=1, \cdots, q},\\
			\mathcal{G}_q(\theta)=&
			\left(
			\frac{1}{4\pi}\int_{-\pi}^\pi\frac{\partial}{\partial\theta_j}\log{g_\theta(\lambda)}\frac{\partial}{\partial\theta_k}\log{g_\theta(\lambda)}\, d\lambda
			\right)_{j, k=1, \cdots, q} .
		\end{align*}
		for $q=1, \cdots, p$, where the function $g_\theta(\lambda)$, $\lambda\in[-\pi, \pi]$ is defined by
		\begin{align}\label{normalized.spd}
			g_\theta(\lambda)=\frac{f_\theta(\lambda)}{b(\theta)}\hspace{0.2cm}\mbox{with}\hspace{0.2cm}b(\theta)=\exp\left(\frac{1}{2\pi}\int_{-\pi}^{\pi}\log{f_\theta(\lambda)}\, d\lambda\right).
		\end{align}
		Moreover, we set $a_0(\theta)=0$ and denote the diffusion parameter $\sigma$ as $\theta_{p+1}$ 
		for notational simplicity. 
		
		\subsection{Only the Hurst parameter $H$ is known}
		Assume the true value of the Hurst parameter $H_0$ is known. 
		Under the condition $(A.\ref{A.identify})$, we set an estimation function $\bar{L}_{N}(\psi, \sigma)$ as follows.
		\begin{align*}
			\bar{L}_N(\psi, \sigma)&=\frac{1}{2\pi} \int_{-\pi}^{\pi}\log\left(\sigma^2\delta^{2H_0}f_{\psi}(\lambda)\right)\, d\lambda 
			+\frac{1}{2\pi} \int_{-\pi}^{\pi} \frac{I_N(\lambda)}{\sigma^2\delta^{2H_0}f_{\psi}(\lambda)}\, d\lambda \\
			&=2H_0\log\delta + \frac{1}{2\pi} \int_{-\pi}^{\pi}\log\left(\sigma^2 f_{\psi}(\lambda)\right)\, d\lambda 
			+ \frac{1}{2\pi} \int_{-\pi}^{\pi} \frac{\bar{I}_N(\lambda)}{\sigma^2 f_{\psi}(\lambda)}\, d\lambda,
		\end{align*}
		where $f_{\psi}(\lambda)=f(\psi, H_0, \lambda)$ and
		\begin{align}\label{periodogram}
			I_N(\lambda) = \frac{1}{2\pi N} {\left| \sum_{j=1}^{N}X_j^{\delta}e^{\sqrt{-1}j\lambda}\right|}^2, \bar{I}_N(\lambda) = \frac{1}{2\pi N} {\left| \sum_{j=1}^{N}\delta^{-H_0}X_j^{\delta}e^{\sqrt{-1}j\lambda}\right|}^2.
		\end{align}
		Note that $\{\delta^{-H_0}{X}_j^\delta\}_{j=1, \cdots, N}$ is independent of the length of the sampling interval $\delta$ 
		because it holds that 
		\begin{align}\label{true.data.scale}
			\left({\delta}^{-H_0}X_1^\delta, \cdots, {\delta}^{-H_0}X_N^\delta\right)=
			\left(X_1^1, \cdots, X_N^1\right)
		\end{align}
		in law under $P_{\theta_0, \sigma_0}^{(\delta)}$ from the condition $(A.\ref{A.identify})$. 
		As a result, we can regard $\{\delta^{-H_0}{X}_j^\delta\}_{j=1, \cdots, N}$ as an stationary Gaussian sequence
		with mean $0$ and spectral density $\sigma^2 f_\theta(\lambda)$. 
		Here, we set an estimator 
		\begin{align*}
			\left(\bar{\psi}_N, \bar{\sigma}_N\right)&=\argmin_{(\psi, H_0, \sigma)\in\Theta\times\Sigma} \bar{L}_N(\psi, \sigma)\\
			&=\argmin_{(\psi, H_0, \sigma)\in\Theta\times\Sigma}\left\{\frac{1}{2\pi} \int_{-\pi}^{\pi}\log\left(\sigma^2 f_{\psi}(\lambda)\right)\, d\lambda 
			+\frac{1}{2\pi} \int_{-\pi}^{\pi} \frac{\bar{I}_N(\lambda)}{\sigma^2 f_{\psi}(\lambda)}\, d\lambda\right\}.
		\end{align*}
		Then, we get the following theorem from Theorem 5 of Lieberman et al. (2009) 
		and a easy modification of Theorem 2.4 of Cohen et al. (2013). 
		\begin{thm}{\label{asy.nor.known}}
			Suppose Assumption \ref{model.assump.} and $H_0$ is known. 
			Then, $(\bar{\psi}_N, \bar{\sigma}_N)$ is an asymptotically efficient estimator in Fisher's sense, that is,
			\begin{align*}
				\mathcal{L}\left\{\sqrt{N}
				\begin{pmatrix}
					\bar{\psi}_N - \psi_0\\
					\bar{\sigma}_N - \sigma_0
				\end{pmatrix}
				\bigg| P_{\theta_0, \sigma_0}^{(\delta)}\right\}\stackrel{\delta\to 0}{\rightarrow}\mathcal{N}\left(0,\mathcal{F}(\psi_0, \sigma_0)^{-1}\right),
			\end{align*}
			where the asymptotic Fisher information matrix $\mathcal{F}(\psi, \sigma)$ is given by
			\begin{align*}
				\mathcal{F}(\psi, \sigma)=\left[\frac{1}{4\pi}\int_{-\pi}^{\pi}\frac{\partial}{\partial\theta_j}\log\left(\sigma^2 f_\psi(\lambda)\right)\frac{\partial}{\partial\theta_k}\log\left(\sigma^2 f_\psi(\lambda)\right)\, d\lambda\right]_{j, k=1, \cdots, p-1, p+1}.
			\end{align*}
		\end{thm}
		
		\subsection{All parameters $(\theta, \sigma)$ are unknown}
		In this section, we consider the case that all parameters $(\theta, \sigma)$ are unknown. 
		This case is more difficult than the previous case because we can not substantially take the appropriate scaling of the data like (\ref{true.data.scale}).
		However, we can construct an asymptotically efficient estimator as follows: 
		At first, the normalized spectral density $g_\theta$, which is defined in (\ref{normalized.spd}), satisfies the conditions $(A.\ref{A.decomposition})-(A.\ref{spd.asyrelation})$ and
		\begin{align} \label{specnormal}
			\frac{1}{2\pi}\int_{-\pi}^{\pi}\log{g_\theta(\lambda)}d\lambda=0 \hspace{0.2cm}\mbox{for all $\theta\in\Theta$}.
		\end{align}
		Therefore, we set an estimation function ${\sigma}^2_N(\theta)$ and an estimator $\hat{\theta}_N$ as
		\begin{align*} 
			{\sigma}_N^2(\theta)=\frac{1}{2\pi}\int_{-\pi}^{\pi} \frac{I_N(\lambda)}{g_\theta(\lambda)}\, d\lambda\hspace{0.2cm},\hspace{0.2cm}\hat{\theta}_N=\argmin_{\theta\in\Theta}{\sigma}_N^2(\theta),
		\end{align*}
		where $I_N(\lambda)$ is defined in (\ref{periodogram}). 
		Note that the estimator $\hat{\theta}_N$ also minimizes an scaled estimation function $\widetilde{\sigma}_N^2(\theta)$:
		\begin{align} 
			\widetilde{\sigma}_N^2(\theta) = (\delta^{2H_0} b(\theta_0))^{-1} {\sigma}_N^2(\theta) = \frac{1}{2\pi} \int_{-\pi}^{\pi} \frac{\widetilde{I}_N(\lambda)}{g_\theta(\lambda)}\, d\lambda, 
		\end{align}
		where
		\begin{align*}
			\widetilde{I}_N(\lambda) = \frac{1}{2\pi N} {\Biggl| \sum_{j=1}^{N}\widetilde{X}_je^{\sqrt{-1}j{\lambda}}\Biggr|}^2\hspace{0.2cm},
			\hspace{0.2cm}\widetilde{X}_j=(\delta^{2H_0} b(\theta_0))^{-\frac{1}{2}}X_j^\delta. 
		\end{align*}
		Then, we can regard $\{\widetilde{X}_j\}_{j=1, \cdots, N}$ as a stationary Gaussian sequence with mean $0$ and spectral density $\sigma^2 g_\theta(\lambda)$ 
		under $P_{\theta_0, \sigma_0}^{(\delta)}$ in the same way as (\ref{true.data.scale}). 
		In other word, we can formally take an appropriate scaling of the data even if $H_0$ is unknown. 
		Set $\widetilde{\sigma}_N=\sqrt{\widetilde{\sigma}_N^2(\hat{\theta}_N)}$. Then, we regard a random variable $(\hat{\theta}_N, \widetilde{\sigma}_N)$  
		as a minimizer of a function $L_N(\theta, \sigma)$, that is, 
		\begin{align*}
			(\hat{\theta}_N, \widetilde{\sigma}_N)=\argmin_{(\theta, \sigma)\in\Theta\times\Sigma} L_N(\theta, \sigma),
		\end{align*}
		where
		\begin{align*}
			L_N(\theta, \sigma)=\log\sigma^2 + \frac{1}{2\pi} \int_{-\pi}^{\pi} \frac{\widetilde{I}_N(\lambda)}{\sigma^2 g_\theta(\lambda)}\, d\lambda. 
		\end{align*}
		Therefore, we get the following central limit theorem (CLT) from Theorem 5 of Lieberman et al. (2009) in the similar way as Theorem \ref{asy.nor.known}.
		\begin{thm} \label{asynormal.LRR}
			Suppose Assumption \ref{model.assump.}. Then, the following CLT holds.
			\begin{align}
				\mathcal{L}\left\{\sqrt{N}
				\begin{pmatrix}
					\hat{\theta}_N - \theta_0 \\
					\widetilde{\sigma}_N -\sigma_0
				\end{pmatrix}
				\bigg| P_{\theta_0, \sigma_0}^{(\delta)}\right\}
				\stackrel{\delta\to 0}{\rightarrow}\mathcal{N}\left(0,
				\begin{bmatrix}
					\mathcal{G}_p({\theta}_0)^{-1}&0_{p\times 1} \\
					0_{1\times p}&\frac{\sigma_0^2}{2}
				\end{bmatrix}
				\right). \label{CLT.scale.}
			\end{align}
		\end{thm}
		
		\begin{proof}
			In the similar way as Theorem \ref{asy.nor.known}, (\ref{CLT.scale.}) holds and 
			the asymptotic covariance matrix is given by the inverse of the matrix
			\begin{align}
				\left[\frac{1}{4\pi}\int_{-\pi}^{\pi}\frac{\partial}{\partial\theta_j}\log\left(\sigma^2 g_\theta(\lambda)\right)\frac{\partial}{\partial\theta_k}\log\left(\sigma^2 g_\theta(\lambda)\right)\, d\lambda\right]_{j, k=1, \cdots, p+1}.\label{asy.var.}
			\end{align}
			Then, the components of the matrix (\ref{asy.var.}) are calculated as
			\begin{align*}
				\frac{1}{4\pi}\int_{-\pi}^{\pi}\left[\frac{\partial}{\partial\sigma}\log\left(\sigma^2 g_\theta(\lambda)\right)\right]^2\, d\lambda &= \frac{2}{\sigma^2},\\
				\int_{-\pi}^{\pi}\frac{\partial}{\partial\sigma}\log\left(\sigma^2 g_\theta(\lambda)\right) \frac{\partial}{\partial\theta_j}\log\left(\sigma^2 g_\theta(\lambda)\right)\, d\lambda &= \frac{2}{\sigma}\frac{\partial}{\partial\theta_j}\int_{-\pi}^{\pi} \log{g_\theta(\lambda)}\, d\lambda = 0, \\
				\int_{-\pi}^{\pi}\frac{\partial}{\partial\theta_j}\log\left(\sigma^2 g_\theta(\lambda)\right) \frac{\partial}{\partial\theta_k}\log\left(\sigma^2 g_\theta(\lambda)\right)\, d\lambda &= \int_{-\pi}^{\pi}\frac{\partial}{\partial\theta_j} \log{g_\theta(\lambda)}\frac{\partial}{\partial\theta_k} \log{g_\theta(\lambda)}\, d\lambda,
			\end{align*}
			for each $j, k = 1, \cdots, p$ from the relation (\ref{specnormal}). 
			Hence, the matrix (\ref{asy.var.}) is equal to 
			\begin{align}\label{asy.var.trans.}
				\begin{bmatrix}
					\mathcal{G}_p({\theta})&0_{p\times 1} \\
					0_{1\times p}&\frac{2}{\sigma^2}
				\end{bmatrix}.
			\end{align}
			Moreover, the inverse of the matrix (\ref{asy.var.trans.}) is calculated as follows.
			\begin{align*}
				\begin{bmatrix}
					\mathcal{G}_p({\theta})&0_{p\times 1} \\
					0_{1\times p}&\frac{2}{\sigma^2}
				\end{bmatrix}^{-1}
				=
				\begin{bmatrix}
					\mathcal{G}_p({\theta})^{-1}&0_{p\times 1} \\
					0_{1\times p}&\frac{\sigma^2}{2}
				\end{bmatrix}.
			\end{align*}
			Therefore, this finishes the proof of Theorem \ref{asynormal.LRR}.
		\end{proof}
		
		Note that the random variable $\widetilde{\sigma}_N$ is not a statistic because we construct this by using the true value $\theta_0$. 
		Therefore, we construct an estimator $\hat{\sigma}_N$ for the diffusion parameter $\sigma$ by substituting the estimators 
		$\hat{\theta}_N$ and $\hat{H}_N$ to the true values $\theta_0$ and $H_0$ in the scaled function $\widetilde{\sigma}_N$. 
		Namely, we set 
		\begin{align*}
			\hat{\sigma}_N=\sqrt{(\delta^{2\hat{H}_N} b(\hat{\theta}_N))^{-1}{\sigma}_N^2(\hat{\theta}_N)}.
		\end{align*}
		Here, we show the estimator $(\hat{\theta}_N, \hat{\sigma}_N)$ is asymptotically efficient in Fisher's sense. 
		Before showing this claim, we review the definition of the asymptotically efficient estimator 
		in this sense more precisely following by Ibragimov and Has'minskii (1981), p.159. 
		
		\begin{defi}\label{defi.eff.}
			Let $\Theta\subset\mathbb{R}^p$ and a family of measures $\{P_\theta^{(\delta)}; \theta\in\Theta\}$ 
			satisfy the LAN property at a point $\theta_0\in\Theta$ as $\delta\to 0$, that is, 
			for a certain $p\times p$-matrix $\phi_\delta(\theta_0)$ and any $u\in\mathbb{R}^{p}$, 
			the log-likelihood ratio admits the representation$:$
			\begin{align*}
				\log{\frac{dP_{\theta_0+\phi_\delta(\theta_0)u}^{(\delta)}}{dP_{\theta_0}^{(\delta)}}}=\langle u, \zeta_\delta(\theta_0)\rangle - \frac{1}{2}\langle \mathcal{J}(\theta_0)u, u\rangle + r_\delta(\theta_0),
			\end{align*}
			where $\mathcal{J}(\theta_0)$ is a nondegenerate $p\times p$-matrix and 
			\begin{align*}
				\zeta_\delta(\theta_0)\rightarrow\mathcal{N}(0, \mathcal{J}(\theta_0)),\hspace{0.2cm} r_\delta(\theta_0)\rightarrow 0,
			\end{align*}
			in law under $P_{\theta_0}^{(\delta)}$ as $\delta\to 0$. A sequence of estimators $\hat{\theta}_N$ is called asymptotically efficient in Fisher's sense 
			at the point $\theta_0$ if it holds that
			\begin{align*}
				\mathcal{L}\left\{\phi_\delta^{-1}(\theta_0)(\hat{\theta}_N-\theta_0)
				\Big|P_{\theta_0}^{(\delta)}\right\}\stackrel{\delta\to 0}{\rightarrow}\mathcal{N}\left(0, \mathcal{J}(\theta_0)^{-1}\right).
			\end{align*}
			Here, the matrix $\phi_\delta(\theta_0)$ is usually called as a rate matrix. 
		\end{defi}
		
		\begin{thm}\label{CLTsigmahat}
			Suppose Assumption \ref{model.assump.} and a rate matrix $\phi_N=\phi_N(\theta_0, \sigma_0)$ 
			satisfies Assumption \ref{assump.rate} in Appendix. 
			Then, the sequence of estimators $(\hat{\theta}_N, \hat{\sigma}_N)$ is asymptotically efficient in Fisher's sense, that is, 
			\begin{align*}
				\mathcal{L}\left\{\phi_N^{-1}
				\begin{pmatrix}
					\hat{\theta}_N-\theta_0\\
					\hat{\sigma}_N-\sigma_0
				\end{pmatrix}
				\bigg|P_{\theta_0, \sigma_0}^{(\delta)}\right\}\stackrel{\delta\to 0}{\rightarrow}\mathcal{N}\left(0, \mathcal{J}(\theta_0, \sigma_0)^{-1}\right),
			\end{align*}
			where
			\begin{align*}
				&\mathcal{J}(\theta, \sigma)=
				\begin{pmatrix}
					D& 0_{p-1\times 2} \\
					0_{2\times p-1}& E
				\end{pmatrix}
				\begin{pmatrix}
					\mathcal{F}_p(\theta)& a_p(\theta) \\
					a_p(\theta)^{\mathrm{T}}& 2
				\end{pmatrix}
				\begin{pmatrix}
					D& 0_{p-1\times 2} \\
					0_{2\times p-1}& E
				\end{pmatrix}^{\mathrm{T}}, \\
				&D=D(\theta_0, \sigma_0)=\mathrm{diag}(d^{(1)}, \cdots, d^{(p-1)}) , 
				\hspace{0.2cm}
				E=E(\theta_0, \sigma_0)=
				\begin{pmatrix}
					\alpha& \gamma \\
					\hat{\alpha}& \hat{\gamma}
				\end{pmatrix}.
			\end{align*}
		\end{thm}
		\begin{proof}
			At first, it holds that
			\begin{align}
				\phi_N^{-1}
				\begin{pmatrix}
					\hat{\theta}_N-\theta_0\\
					\hat{\sigma}_N-\sigma_0
				\end{pmatrix}
				&=
				\begin{pmatrix}
					\phi_{N, 1}^{-1}(\hat{\psi}_N-\psi_0) \\
					\phi_{N, 2}^{-1}
					\begin{pmatrix}
						\hat{H}_N-H_0 \\
						\hat{\sigma}_N-\sigma_0
					\end{pmatrix}
				\end{pmatrix}, \label{invm.decomp.}
			\end{align}
			where $\phi_{N, 1}^{-1}$, $\phi_{N, 2}^{-1}$ are inverse matrices of $\phi_{N, 1}$, $\phi_{N, 2}$ respectively, that is, 
			\begin{align*}
				\phi_{N, 1}^{-1}=\mathrm{diag}\left(\frac{1}{d_N^{(1)}}, \cdots, \frac{1}{d_N^{(p-1)}}\right), \phi_{N, 2}^{-1}=\frac{1}{\mathrm{det}(\phi_{N, 2})}
				\begin{bmatrix}
					\hat{\beta}_N&-\hat{\alpha}_N \\
					-\beta_N&\alpha_N
				\end{bmatrix}.
			\end{align*}
			Here, $\mathrm{det}(\phi_{N, 2})$ is calculated as 
			\begin{align*}
				&\frac{\sigma_0}{N}\left(\sqrt{N}\alpha_N\frac{\sqrt{N}\hat{\beta}_N}{\sigma_0}-\sqrt{N}\hat{\alpha}_N\frac{\sqrt{N}\beta_N}{\sigma_0}\right) \\
				=&\frac{\sigma_0}{N}\left\{\sqrt{N}\alpha_N\left(\hat{\alpha}_N\sqrt{N}\log\delta+\frac{\sqrt{N}\hat{\beta}_N}{\sigma_0}\right)
				-\sqrt{N}\hat{\alpha}_N\left(\sqrt{N}\alpha_N\log\delta+\frac{\sqrt{N}\beta_N}{\sigma_0}\right)\right\} \\
				=&\frac{\sigma_0}{N}\left(\sqrt{N}\alpha_N\hat{\gamma}_N-\sqrt{N}\hat{\alpha}_N\gamma_N\right).
			\end{align*}
			Set $\acute{\sigma}_N=\sqrt{(\delta^{2H_0} b(\hat{\theta}_N))^{-1}{\sigma}_N^2(\hat{\theta}_N)}$. 
			Then, the estimation error of $\hat{\sigma}_N-\sigma_0$ is expanded as follows by using the Taylor's formula.
			\begin{align}
				&\hat{\sigma}_N-\sigma_0 \nonumber\\
				=&\sigma_0\left(\log\hat{\sigma}_N-\log\sigma_0\right)+o_{P_{\theta_0, \sigma_0}^{(\delta)}}(N^{-\frac{1}{2}}) \nonumber\\
				=&\sigma_0\left\{\log\acute{\sigma}_N-\log\sigma_0\right\}
				-\sigma_0\log\delta\left(\hat{H}_N-H_0\right)+o_{P_{\theta_0, \sigma_0}^{(\delta)}}(N^{-\frac{1}{2}})\nonumber\\
				=&\acute{\sigma}_N-\sigma_0-\sigma_0\log\delta\left(\hat{H}_N-H_0\right)+o_{P_{\theta_0, \sigma_0}^{(\delta)}}(N^{-\frac{1}{2}}). \label{asyexpsigma}
			\end{align}
			Therefore, the following asymptotic expansion holds. 
			\begin{align*}
				&\phi_{N, 2}^{-1}
				\begin{pmatrix}
					\hat{H}_N-H_0 \\
					\hat{\sigma}_N-\sigma_0
				\end{pmatrix}\\
				=&\frac{1}{\mathrm{det}(\phi_{N, 2})}
				\begin{pmatrix}
					\hat{\beta}_N(\hat{H}_N-H_0) -\hat{\alpha}_N(\hat{\sigma}_N-\sigma_0) \\
					-\beta_N(\hat{H}_N-H_0) +\alpha_N(\hat{\sigma}_N-\sigma_0)
				\end{pmatrix}\\
				=&\frac{1}{\mathrm{det}(\phi_{N, 2})}
				\begin{pmatrix}
					(\hat{\beta}_N+\sigma_0\hat{\alpha}_N\log\delta)(\hat{H}_N-H_0) -\hat{\alpha}_N(\acute{\sigma}_N-\sigma_0) \\
					-(\beta_N+\sigma_0\alpha_N\log\delta)(\hat{H}_N-H_0) +\alpha_N(\acute{\sigma}_N-\sigma_0)
				\end{pmatrix}
				+o_{P_{\sigma_0, \theta_0}^{(\delta)}}(1) \\
				=&\frac{\sigma_0}{N\mathrm{det}(\phi_{N, 2})}
				\begin{pmatrix}
					\hat{\gamma}_N\sqrt{N}(\hat{H}_N-H_0) -\sqrt{N}\hat{\alpha}_N\frac{\sqrt{N}}{\sigma_0}(\acute{\sigma}_N-\sigma_0) \\
					-\gamma_N\sqrt{N}(\hat{H}_N-H_0) +\sqrt{N}\alpha_N\frac{\sqrt{N}}{\sigma_0}(\acute{\sigma}_N-\sigma_0)
				\end{pmatrix}
				+o_{P_{\sigma_0, \theta_0}^{(\delta)}}(1)\\
				=&\frac{\sigma_0}{N\mathrm{det}(\phi_{N, 2})}
				\begin{pmatrix}
					\hat{\gamma}_N& -\sqrt{N}\hat{\alpha}_N \\
					-\gamma_N& \sqrt{N}\alpha_N
				\end{pmatrix}
				\begin{pmatrix}
					\sqrt{N}(\hat{H}_N-H_0) \\
					\frac{\sqrt{N}}{\sigma_0}(\acute{\sigma}_N-\sigma_0) 
				\end{pmatrix}
				+o_{P_{\sigma_0, \theta_0}^{(\delta)}}(1) \\
				=&\left(E_N^{\mathrm{T}}\right)^{-1}
				\begin{pmatrix}
					\sqrt{N}(\hat{H}_N-H_0) \\
					\frac{\sqrt{N}}{\sigma_0}(\acute{\sigma}_N-\sigma_0) 
				\end{pmatrix}
				+o_{P_{\sigma_0, \theta_0}^{(\delta)}}(1),
			\end{align*}
			where
			\begin{align*}
				E_N=E_N(\theta_0, \sigma_0)=
				\begin{pmatrix}
					\sqrt{N}\alpha_N& \gamma_N \\
					\sqrt{N}\hat{\alpha}_N& \hat{\gamma}_N
				\end{pmatrix}.
			\end{align*}
			Note that we use (\ref{asyexpsigma}) and $\mathrm{det}(\phi_{N, 2})=O(N)$ in the second equality. 
			From this asymptotic expansion and (\ref{invm.decomp.}), it holds that
			\begin{align*}
				&\phi_N^{-1}
				\begin{pmatrix}
					\hat{\theta}_N-\theta_0\\
					\hat{\sigma}_N-\sigma_0
				\end{pmatrix}
				=&
				\begin{pmatrix}
					\sqrt{N}\phi_{N, 1}& 0_{p-1\times 2} \\
					0_{2\times p-1}& E_N^{\mathrm{T}}
				\end{pmatrix}^{-1}
				\begin{pmatrix}
					\sqrt{N}(\hat{\theta}_N-\theta_0) \\
					\frac{\sqrt{N}}{\sigma_0}(\acute{\sigma}_N-\sigma_0) 
				\end{pmatrix}
				+o_{P_{\sigma_0, \theta_0}^{(\delta)}}(1).
			\end{align*}
			Then, from Lemma \ref{acutesigmaCLT} in Appendix, $E_N\to E$, $\sqrt{N}\phi_{N, 1}\to D$ in matrix norm as $\delta\to 0$ 
			and the continuous mapping theorem, 
			it converges to a normal distribution with mean vector $0_{p+1\times 1}$ and covariance matrix given by the inverse matrix of 
			\begin{align*}
				&
				\begin{pmatrix}
					D& 0_{p-1\times 2} \\
					0_{2\times p-1}& E^{\mathrm{T}}
				\end{pmatrix}^{\mathrm{T}}
				\begin{pmatrix}
					\mathcal{F}_p(\theta)& a_p(\theta) \\
					a_p(\theta)^{\mathrm{T}}& 2
				\end{pmatrix}
				\begin{pmatrix}
					D& 0_{p-1\times 2} \\
					0_{2\times p-1}& E^{\mathrm{T}}
				\end{pmatrix} \\
				=&
				\begin{pmatrix}
					D& 0_{p-1\times 2} \\
					0_{2\times p-1}& E
				\end{pmatrix}
				\begin{pmatrix}
					\mathcal{F}_p(\theta)& a_p(\theta) \\
					a_p(\theta)^{\mathrm{T}}& 2
				\end{pmatrix}
				\begin{pmatrix}
					D& 0_{p-1\times 2} \\
					0_{2\times p-1}& E
				\end{pmatrix}^{\mathrm{T}}.
			\end{align*}
			The conclusion follows from the above CLT and Theorem \ref{thm.LAN} in Appendix.
		\end{proof}
		
		\subsection{Only the diffusion parameter $\sigma$ is known}
		Our purpose in this section is to construct an asymptotically efficient estimator for the parameter $\theta=(\psi, H)$ 
		in Fisher's sense in the case that the true value of the diffusion parameter $\sigma_0$ is known. 
		Namely, we construct an estimator $(\hat{\psi}_N, \hat{H}_N)$ for the parameter $(\psi, H)$ satisfying the following relation:
		\begin{align}\label{eff.H.sigmaknown}
			&\mathcal{L}\left\{
			\begin{pmatrix}
				\sqrt{N}(\hat{\psi}_N - \psi_0)\\
				\sqrt{N}|\log\delta|(\hat{H}_N - H_0)
			\end{pmatrix}
			\Bigg| P_{\theta_0, \sigma_0}^{(\delta)}\right\}
			\stackrel{\delta\to 0}{\rightarrow}\mathcal{N}\left(0, \mathcal{I}(\theta_0)^{-1}\right),
		\end{align}
		where 
		\begin{align*}
			\mathcal{I}(\theta)^{-1}=
			\begin{bmatrix}
				\mathcal{G}_{p-1}(\theta)^{-1}&-\frac{1}{2}\mathcal{G}_{p-1}(\theta)^{-1}a_{p-1}(\theta) \\
				-\frac{1}{2}a_{p-1}(\theta)^{\mathrm{T}}\mathcal{G}_{p-1}(\theta)^{-1}&\frac{1}{2}+\frac{1}{4}a_{p-1}(\theta)^{\mathrm{T}}\mathcal{G}_{p-1}(\theta)^{-1}a_{p-1}(\theta)
			\end{bmatrix}.
		\end{align*}
		In fact, the relation (\ref{eff.H.sigmaknown}) means the asymptotic efficiency in Fisher's sense, 
		which is derived by Lemma \ref{mat.eq.} and Theorem \ref{Kawai.extension} in Appendix.
		Note that the optimal asymptotic variance $\mathcal{I}(\theta)^{-1}$ is independent of the estimation error 
		of the Hurst parameter as well as that of the diffusion parameter.  
		
		At first, we work under a strong assumption that a value $ b(\theta_0)$ is known.
		Then, we define an estimator $\hat{H}_N^{(0)}$ as the solution of the equation  
		\begin{align*}
			\log{\frac{\sigma^2(\hat{\theta}_N)}{\delta^{2H}\beta(\theta_0)}}-\log\sigma_0^2=0
		\end{align*}
		with respect to $H$. More precisely, we set the estimator $\hat{H}_N^{(0)}$ as
		\begin{align*}
			\hat{H}_N^{(0)}=\frac{1}{2|\log\delta|}\left\{\log b(\theta_0)-\log{\sigma_N^2(\hat{\theta}_N)}+\log\sigma_0^2\right\},
		\end{align*}
		where $\hat{\theta}_N=(\hat{\psi}_N, \hat{H}_N)=(\hat{\psi}^{(1)}_N, \cdots, \hat{\psi}^{(p-1)}_N, \hat{H}_N)$ is the Whittle estimator defined in the previous section. 
		Then, we get the following theorem. 
		\begin{thm} \label{asynormal.pro.}
			Suppose Assumption \ref{model.assump.} and $ b(\theta_0)$ is known. 
			Then, $(\hat{\theta}_N, \hat{H}_N^{(0)})$ holds the following CLT:
			\begin{align*}
				&\mathcal{L}\left\{
				\begin{pmatrix}
					\sqrt{N}(\hat{\theta}_N - \theta_0)\\
					\sqrt{N}|\log\delta|(\hat{H}_N^{(0)} - H_0)
				\end{pmatrix}
				\Bigg| P_{\theta_0, \sigma_0}^{(\delta)}\right\}
				\stackrel{\delta\to 0}{\rightarrow}\mathcal{N}\left(0, \mathcal{A}^{(0)}(\theta_0)
				\right),
			\end{align*}
			where 
			\begin{align*}
				\mathcal{A}^{(0)}(\theta)=
				\begin{bmatrix}
					I_{p}&0_{p\times 1} \\
					0_{1\times p}&-\frac{1}{\sigma}
				\end{bmatrix}
				\begin{bmatrix}
					\mathcal{G}_p(\theta)^{-1}&0_{p\times 1} \\
					0_{1\times p}&\frac{\sigma^2}{2}
				\end{bmatrix}
				\begin{bmatrix}
					I_{p}&0_{p\times 1} \\
					0_{1\times p}&-\frac{1}{\sigma}
				\end{bmatrix}^{\mathrm{T}}
				=\begin{bmatrix}
					\mathcal{G}_p(\theta)^{-1}&0 \\
					0&\frac{1}{2}
				\end{bmatrix}.
			\end{align*}
			In particular, the estimator $(\hat{\psi}_N, \hat{H}_N^{(0)})$ is asymptotically normal. 
		\end{thm}
		\begin{proof}
			From the definition of $\hat{H}_N^{(0)}$ and the delta-method, it holds that
			\begin{align}
				\sqrt{N}|\log{\delta}|\left(\hat{H}_N^{(0)}-H_0\right)&=-\frac{\sqrt{N}}{2}\left\{\log\frac{\sigma_N^2(\hat{\theta}_N)}{\delta^{2H_0} b(\theta_0)}-\log\sigma_0^2\right\} \nonumber\\
				&=-\frac{\sqrt{N}}{\sigma_0}\left(\widetilde{\sigma}_N-{\sigma}_0\right)+o_{P_{\theta_0, \sigma_0}^{(\delta)}}(1). \label{Hurst0.delta}
			\end{align}
			Then, the following asymptotic expansion holds. 
			\begin{align*}
				\begin{pmatrix}
					\sqrt{N}(\hat{\theta}_N - \theta_0)\\
					\sqrt{N}|\log\delta|(\hat{H}_N^{(0)} - H_0)
				\end{pmatrix}
				&=
				\begin{bmatrix}
					I_{p}&0_{p\times 1} \\
					0_{1\times p}&-\frac{1}{\sigma_0}
				\end{bmatrix}
				\cdot\sqrt{N}
				\begin{pmatrix}
					\hat{\theta}_N-\theta_0 \\
					\widetilde{\sigma}_N-{\sigma}_0
				\end{pmatrix}
				+o_{P_{\theta_0, \sigma_0}^{(\delta)}}(1).
			\end{align*}
			Therefore, the first claim follows from Theorem \ref{asynormal.LRR} and the continuous mapping theorem. 
			Moreover, the second claim also follows from the first one and the continuous mapping theorem.
		\end{proof}
		
		The assumption that $ b(\theta_0)$ is known does not hold even in the case of the fractional Gaussian noise 
		without the appropriate selection of the parameter space satisfying the relation (\ref{specnormal}). 
		Theorem \ref{asynormal.pro.}, however, gives an idea to construct estimators 
		because the convergence rate and the asymptotic variance 
		of the estimator $\hat{H}_N^{(0)}$ achieve the efficient ones in the case of the fractional Gaussian noise derived in Kawai (2013). 
		Now, we remove the assumption that $ b(\theta_0)$ is known. 
		Then, we consider to substitute the Whittle estimator $\hat{\theta}_N$ to the true value $\theta_0$ in $ b(\theta_0)$. 
		Namely, we define an estimator $\hat{H}_N^{(1)}$ as
		\begin{align*}
			&\hat{H}_N^{(1)}=\frac{1}{2|\log\delta|}\left\{\log b(\hat{\theta}_N)-\log{\sigma_N^2(\hat{\theta}_N)}+\log\sigma_0^2\right\}. 
		\end{align*}
		Then, we can prove asymptotic normality of the estimator $(\hat{\psi}_N, \hat{H}_N^{(1)})$ as follows. 
		\begin{thm} \label{asympnormal.H}
			Suppose Assumption \ref{model.assump.}. Then, $(\hat{\theta}_N, \hat{H}_N^{(1)})$ holds the following CLT:
			\begin{align*}
				&\mathcal{L}\left\{
				\begin{pmatrix}
					\sqrt{N}(\hat{\theta}_N - \theta_0)\\
					\sqrt{N}|\log\delta|(\hat{H}_N^{(1)} - H_0)
				\end{pmatrix}
				\Bigg| P_{\theta_0, \sigma_0}^{(\delta)}\right\}
				\stackrel{\delta\to 0}{\rightarrow}\mathcal{N}\left(0, \mathcal{A}^{(1)}(\theta_0)
				\right),
			\end{align*}
			where 
			\begin{align}\label{asympnormal.H.var.}
				\mathcal{A}^{(1)}(\theta)&=
				\begin{bmatrix}
					\mathcal{G}_p(\theta)^{-1}&-\frac{1}{2}\mathcal{G}_p(\theta)^{-1}a_p(\theta) \\
					-\frac{1}{2}a_p(\theta)^{\mathrm{T}}\mathcal{G}_p(\theta)^{-1}&\frac{1}{2}+\frac{1}{4}a_p(\theta)^{\mathrm{T}}\mathcal{G}_p(\theta)^{-1}a_p(\theta)
				\end{bmatrix}.
			\end{align}
			In particular, the estimator $(\hat{\psi}_N, \hat{H}_N^{(1)})$ is asymptotically normal. 
		\end{thm}
		
		\begin{proof}
			In the same way as (\ref{Hurst0.delta}), it holds that
			\begin{align*}
				\sqrt{N}|\log{\delta}|\left(\hat{H}_N^{(1)}-H_0\right)
				=&-\frac{\sqrt{N}}{2}\left\{\log\frac{\sigma_N^2(\hat{\theta}_N)}{\delta^{2H_0} b(\hat{\theta}_N)}-\log\sigma_0^2\right\} \\
				=&-\frac{\sqrt{N}}{\sigma_0}\left(\sqrt{\frac{\sigma_N^2(\hat{\theta}_N)}{\delta^{2H_0} b(\hat{\theta}_N)}}-{\sigma}_0\right)+o_{P_{\theta_0, \sigma_0}^{(\delta)}}(1).
			\end{align*}
			The conclusions follow from Lemma \ref{acutesigmaCLT} and 
			the similar arguments as the last of the proof in Theorem \ref{asynormal.pro.}.
		\end{proof}
		
		Unfortunately, the estimator $(\hat{\psi}_N, \hat{H}_N^{(1)})$ is not asymptotically efficient 
		because the asymptotic covariance matrix of this estimator depends on the estimation error of the Hurst parameter; 
		compare (\ref{eff.H.sigmaknown}) with (\ref{asympnormal.H.var.}).
		However, we can construct an asymptotically efficient estimator by using the estimator $(\hat{\psi}_N, \hat{H}_N^{(1)})$. 
		Namely, we define an estimator $\hat{H}_N^{(2)}$ as
		\begin{align*}
			\hat{H}_N^{(2)}=\frac{1}{2|\log\delta|}\left\{\log b(\hat{\theta}_N^{(1)})-\log{\sigma_N^2(\hat{\theta}_N)}+\log\sigma_0^2\right\}, 
		\end{align*}
		where $\hat{\theta}_N^{(1)}=(\hat{\psi}_N, \hat{H}_N^{(1)})$. 
		Then, we can prove the asymptotic efficiency of the estimator $(\hat{\psi}_N, \hat{H}_N^{(2)})$ as follows.
		\begin{thm}
			Suppose Assumption \ref{model.assump.}. 
			Then, the estimator $(\hat{\psi}_N, \hat{H}_N^{(2)})$  is asymptotically efficient in Fisher's sense, that is, 
			the estimator $(\hat{\psi}_N, \hat{H}_N^{(2)})$ satisfies the relation (\ref{eff.H.sigmaknown}).
		\end{thm}
		
		\begin{proof}
			In the similar way as (\ref{Hurst0.delta}), it holds that
			\begin{align}
				&\sqrt{N}|\log{\delta}|\left(\hat{H}_N^{(2)}-H_0\right) \nonumber\\
				=&-\frac{\sqrt{N}}{\sigma_0}\left(\widetilde{\sigma}_N-{\sigma}_0\right)+\frac{\sqrt{N}}{2}\left(\log b(\hat{\theta}_N^{(1)})-\log b(\theta_0)\right)
				+o_{P_{\theta_0, \sigma_0}^{(\delta)}}(1). \label{Hurst2.error}
			\end{align}
			Then, the second term in (\ref{Hurst2.error}) is expanded as follows. 
			\begin{align}
				&\frac{\sqrt{N}}{2}\left(\log b(\hat{\theta}_N^{(1)})-\log b(\theta_0)\right) \nonumber\\
				=&-\frac{1}{2}a_p(\theta_0)^{\mathrm{T}}\sqrt{N}\left(\hat{\theta}_N^{(1)}-\theta_0\right)+o_{P_{\theta_0, \sigma_0}^{(\delta)}}(1) \nonumber\\
				=&-\frac{1}{2}a_{p-1}(\theta_0)^{\mathrm{T}}\sqrt{N}\left(\hat{\psi}_N-\psi_0\right)+o_{P_{\theta_0, \sigma_0}^{(\delta)}}(1). \label{Hurst2.putinerror}
			\end{align}
			Note that we use the $\sqrt{N}|\log\delta|$-consistency of the estimator $\hat{H}_N^{(1)}$ in the second equality. 
			Therefore, the following asymptotic expansion follows from the relations (\ref{Hurst2.error}) and (\ref{Hurst2.putinerror}).
			\begin{align*}
				\begin{pmatrix}
					\sqrt{N}(\hat{\psi}_N - \psi_0)\\
					\sqrt{N}|\log\delta|(\hat{H}_N^{(2)} - H_0)
				\end{pmatrix}
				&=
				\begin{bmatrix}
					I_{p-1}&0_{p\times 1} \\
					-\frac{1}{2}a_{p-1}(\theta_0)^{\mathrm{T}}&-\frac{1}{\sigma_0}
				\end{bmatrix}
				\cdot\sqrt{N}
				\begin{pmatrix}
					\hat{\psi}_N-\psi_0 \\
					\widetilde{\sigma}_N-{\sigma}_0
				\end{pmatrix}
				+o_{P_{\theta_0, \sigma_0}^{(\delta)}}(1).
			\end{align*}
			The conclusion follows from an easy modification of Theorem \ref{asynormal.LRR} and Lemma \ref{mat.eq.}, 
			Theorem \ref{Kawai.extension} in Appendix. 
		\end{proof}

		\appendix
		\section{Preliminary Lemmas}
		In this Appendix, we show several lemmas used in the proof of main results. 
		\begin{lem}\label{FG.relation}
			Suppose Assumption \ref{model.assump.}. 
			For $q=1, \cdots, p$, the matrices $\mathcal{F}_q(\theta)$ and $\mathcal{G}_q(\theta)$ are connected with the following relation.
			\begin{align*}
				\mathcal{G}_q(\theta)=\mathcal{F}_q(\theta)-\frac{1}{2}a_q(\theta)a_q(\theta)^{\mathrm{T}}.
			\end{align*}
		\end{lem}
		\begin{proof}
			At first, it holds that
			\begin{align*}
				\frac{\partial}{\partial\theta_j}\log{g_\theta(\lambda)}=&\frac{\partial}{\partial\theta_j}\log{f_\theta(\lambda)}-\frac{\partial}{\partial\theta_j}\log\beta(\theta) \\
				=&\frac{\partial}{\partial\theta_j}\log{f_\theta(\lambda)}-\frac{1}{2\pi}\int_{-\pi}^\pi\frac{\partial}{\partial\theta_j}\log{f_\theta(\lambda)}\, d\lambda.
			\end{align*}
			Then, $(j, k)$-component of the matrix $\mathcal{G}_q(\theta)$ is calculated as 
			\begin{align*}
				&\frac{1}{4\pi}\int_{-\pi}^\pi\frac{\partial}{\partial\theta_j}\log{g_\theta(\lambda)}\frac{\partial}{\partial\theta_k}\log{g_\theta(\lambda)}\, d\lambda \\
				=&\frac{1}{4\pi}\int_{-\pi}^\pi\frac{\partial}{\partial\theta_j}\log{f_\theta(\lambda)}\frac{\partial}{\partial\theta_k}\log{f_\theta(\lambda)}\, d\lambda \\
				&-\left(\frac{1}{2\pi}\int_{-\pi}^\pi\frac{\partial}{\partial\theta_j}\log{f_\theta(\lambda)}\, d\lambda\right)\left(\frac{1}{2\pi}\int_{-\pi}^\pi\frac{\partial}{\partial\theta_k}\log{f_\theta(\lambda)}\, d\lambda\right) \\
				&+\frac{1}{2}\left(\frac{1}{2\pi}\int_{-\pi}^\pi\frac{\partial}{\partial\theta_j}\log{f_\theta(\lambda)}\, d\lambda\right)\left(\frac{1}{2\pi}\int_{-\pi}^\pi\frac{\partial}{\partial\theta_k}\log{f_\theta(\lambda)}\, d\lambda\right) \\
				=&\frac{1}{4\pi}\int_{-\pi}^\pi\frac{\partial}{\partial\theta_j}\log{f_\theta(\lambda)}\frac{\partial}{\partial\theta_k}\log{f_\theta(\lambda)}\, d\lambda \\
				&-\frac{1}{2}\left(\frac{1}{2\pi}\int_{-\pi}^\pi\frac{\partial}{\partial\theta_j}\log{f_\theta(\lambda)}\, d\lambda\right)\left(\frac{1}{2\pi}\int_{-\pi}^\pi\frac{\partial}{\partial\theta_k}\log{f_\theta(\lambda)}\, d\lambda\right).
			\end{align*}
			As the above calculations, we finish the proof of this lemma.
		\end{proof}
		
		\begin{lem}\label{mat.eq.}
			Suppose Assumption \ref{model.assump.}. Then, it holds that 
			\begin{align*}
				\begin{bmatrix}
					\mathcal{F}_q(\theta)& a_q(\theta) \\
					a_q(\theta)^{\mathrm{T}}& 2
				\end{bmatrix}^{-1} 
				=
				\begin{bmatrix}
					\mathcal{G}_q(\theta)^{-1}& -\frac{1}{2}\mathcal{G}_q(\theta)^{-1}a_q(\theta) \\
					-\frac{1}{2}a_q(\theta)^{\mathrm{T}}\mathcal{G}_q(\theta)^{-1}& \frac{1}{2}+\frac{1}{4}a_q(\theta)^{\mathrm{T}}\mathcal{G}_q(\theta)^{-1}a_q(\theta)
				\end{bmatrix}
			\end{align*}
			for all $q=1, \cdots, p$.
		\end{lem}
		\begin{proof}
			For our purpose, it suffices to prove the following matrix
			\begin{align}\label{inv.matrixF}
				\begin{bmatrix}
					\mathcal{F}_q(\theta)& a_q(\theta) \\
					a_q(\theta)^{\mathrm{T}}& 2
				\end{bmatrix}
				\begin{bmatrix}
					\mathcal{G}_q(\theta)^{-1}& -\frac{1}{2}\mathcal{G}_q(\theta)^{-1}a_q(\theta) \\
					-\frac{1}{2}a_q(\theta)^{\mathrm{T}}\mathcal{G}_q(\theta)^{-1}& \frac{1}{2}+\frac{1}{4}a_q(\theta)^{\mathrm{T}}\mathcal{G}_q(\theta)^{-1}a_q(\theta)
				\end{bmatrix}
			\end{align}
			is equal to the unit matrix for all $q=1, \cdots, p$. Now, we take on arbitrary $q$. 
			Then, we calculate the matrix (\ref{inv.matrixF}) by using the block matrix and the following relation, 
			which is derived by Lemma \ref{FG.relation},
			\begin{align*}
				\mathcal{F}_q(\theta)\mathcal{G}_q(\theta)^{-1}=I_q+\frac{1}{2}a_q(\theta)a_q(\theta)^{\mathrm{T}}\mathcal{G}_q(\theta)^{-1}.
			\end{align*}
			At first, the $(1, 1)$-component of (\ref{inv.matrixF}) is calculated as 
			\begin{align*}
				\mathcal{F}_q(\theta)\mathcal{G}_q(\theta)^{-1}-\frac{1}{2}a_q(\theta)a_q(\theta)^{\mathrm{T}}\mathcal{G}_q(\theta)^{-1}=I_q.
			\end{align*}
			Next, the $(1, 2)$-component of (\ref{inv.matrixF}) is calculated as 
			\begin{align*}
				-\frac{1}{2}\mathcal{F}_q(\theta)\mathcal{G}_q(\theta)^{-1}a_q(\theta)+\frac{1}{2}a_q(\theta)+\frac{1}{4}a_q(\theta)a_q(\theta)^{\mathrm{T}}\mathcal{G}_q(\theta)^{-1}a_q(\theta)
				=0_{q\times 1}.
			\end{align*}
			Finally, the $(2, 2)$-component of (\ref{inv.matrixF}) is calculated as 
			\begin{align*}
				-\frac{1}{2}a_q(\theta)^{\mathrm{T}}\mathcal{G}_q(\theta)^{-1}a_q(\theta)+1+\frac{1}{2}a_q(\theta)^{\mathrm{T}}\mathcal{G}_q(\theta)^{-1}a_q(\theta)=1.
			\end{align*}
			Therefore, this Lemma follows from the above calculations. 
		\end{proof}
		Set a random variable 
		\begin{align*}
			\acute{\sigma}_N=\sqrt{(\delta^{2H_0}\beta(\hat{\theta}_N))^{-1}{\sigma}_N^2(\hat{\theta}_N)}.
		\end{align*}
		Then, the following CLT holds.
		\begin{lem}\label{acutesigmaCLT}
			Suppose Assumption \ref{model.assump.}. Then, the following CLT holds.
			\begin{align*}
				\mathcal{L}\left\{
				\begin{pmatrix}
					\sqrt{N}(\hat{\theta}_N - \theta_0) \\
					\frac{\sqrt{N}}{\sigma_0}(\acute{\sigma}_N -\sigma_0)
				\end{pmatrix}
				\bigg| P_{\theta_0, \sigma_0}^{(\delta)}\right\}
				\stackrel{\delta\to 0}{\rightarrow}\mathcal{N}\left(0,\mathcal{A}^{(1)}(\theta)\right),
			\end{align*}
			where 
			\begin{align*}
				\mathcal{A}^{(1)}(\theta)&=
				\begin{bmatrix}
					I_{p}&0_{p\times 1} \\
					-\frac{1}{2}a_p(\theta)^{\mathrm{T}}&\frac{1}{\sigma}
				\end{bmatrix}
				\begin{bmatrix}
					\mathcal{G}_p(\theta)^{-1}&0_{p\times 1} \\
					0_{1\times p}&\frac{\sigma^2}{2}
				\end{bmatrix}
				\begin{bmatrix}
					I_{p}&0_{p\times 1} \\
					-\frac{1}{2}a_p(\theta)^{\mathrm{T}}&\frac{1}{\sigma}
				\end{bmatrix}^{\mathrm{T}} \\
				&=
				\begin{bmatrix}
					\mathcal{G}_p(\theta)^{-1}&-\frac{1}{2}\mathcal{G}_p(\theta)^{-1}a_p(\theta) \\
					-\frac{1}{2}a_p(\theta)^{\mathrm{T}}\mathcal{G}_p(\theta)^{-1}&\frac{1}{2}+\frac{1}{4}a_p(\theta)^{\mathrm{T}}\mathcal{G}_p(\theta)^{-1}a_p(\theta)
				\end{bmatrix}
				=
				\begin{bmatrix}
					\mathcal{F}_p(\theta)& a_p(\theta) \\
					a_p(\theta)^{\mathrm{T}}& 2
				\end{bmatrix}^{-1}.
			\end{align*}
		\end{lem}
		\begin{proof}
			From the delta-method, it holds that
			\begin{align}
				&\frac{\sqrt{N}}{\sigma_0}(\acute{\sigma}_N -\sigma_0) \nonumber\\
				=&\sqrt{N}\left(\log\acute{\sigma}_N-\log\sigma_0\right)+o_{P_{\theta_0, \sigma_0}^{(\delta)}}(1) \nonumber\\
				=&\sqrt{N}(\log\widetilde{\sigma}_N-\log\sigma_0)
				-\frac{\sqrt{N}}{2}(\log\beta(\hat{\theta}_N)-\log\beta(\theta_0))+o_{P_{\theta_0, \sigma_0}^{(\delta)}}(1)\nonumber\\
				=&\frac{\sqrt{N}}{\sigma_0}(\widetilde{\sigma}_N-\sigma_0)-\frac{\sqrt{N}}{2}(\log\beta(\hat{\theta}_N)-\log\beta(\theta_0))+o_{P_{\theta_0, \sigma_0}^{(\delta)}}(1). \label{Hurst1.error}
			\end{align}
			Therefore, we need to check the error of the second term in (\ref{Hurst1.error}). From the delta method and the chain rule of composite functions, it holds that
			\begin{align}
				\log\beta(\hat{\theta}_N)-\log\beta(\theta_0) 
				=&\frac{1}{2\pi}\int_{-\pi}^{\pi}\log{f_{\hat{\theta}_N}(\lambda)}\, d\lambda -\frac{1}{2\pi}\int_{-\pi}^{\pi}\log{f_{\theta_0}(\lambda)}\, d\lambda \nonumber \\
				=&a_p(\theta_0)^{\mathrm{T}}\left(\hat{\theta}_N-\theta_0\right)+o_{P_{\theta_0, \sigma_0}^{(\delta)}}(N^{-\frac{1}{2}}). \label{Hurst1.suberr}
			\end{align}
			Therefore, we get the following asymptotic expansion from the relations (\ref{Hurst1.error}), (\ref{Hurst1.suberr}) and Theorem \ref{asynormal.LRR}.
			\begin{align*}
				\begin{pmatrix}
					\sqrt{N}(\hat{\theta}_N - \theta_0) \\
					\frac{\sqrt{N}}{\sigma_0}(\acute{\sigma}_N -\sigma_0)
				\end{pmatrix}
				&=
				\begin{bmatrix}
					I_{p}&0_{p\times 1} \\
					-\frac{1}{2}a_p(\theta_0)^{\mathrm{T}}&\frac{1}{\sigma_0}
				\end{bmatrix}
				\cdot\sqrt{N}
				\begin{pmatrix}
					\hat{\theta}_N-\theta_0 \\
					\widetilde{\sigma}_N-{\sigma}_0
				\end{pmatrix}
				+o_{P_{\theta_0, \sigma_0}^{(\delta)}}(1).
			\end{align*}
			Therefore, we finish the proof from the continuous mapping theorem, 
			Theorem \ref{asynormal.LRR} and Lemma \ref{mat.eq.}.
		\end{proof}
		
		\section{LAN Property under High Frequency Observations}
		In this Appendix, we show several extensions of the results in Kawai (2013) and Brouste and Fukasawa (2016) 
		into our model framework without proof. These results can be proved in the similar arguments. 
		At first, we show the extension of the result in Kawai (2013). Set a rate matrix $\bar{\phi}_N(\theta_0, \sigma_0)$ as follows.
		\begin{align*}
			\bar{\phi}_N(\theta_0, \sigma_0)=
			\begin{pmatrix}
				\frac{1}{\sqrt{N}}I_{p-1}& 0_{p-1\times 1}& 0_{p-1\times 1} \\
				0_{1\times p-1}& \frac{1}{\sqrt{N}\log\delta}& 0 \\
				0_{1\times p-1}& 0& \frac{1}{\sqrt{N}}
			\end{pmatrix}.
		\end{align*}
		Then, we obtain the following weak LAN property. 
		\begin{thm}\label{Kawai.extension}
			Suppose Assumption \ref{model.assump.}. The family of measures $\{P_{\theta, \sigma}^{(\delta)}; (\theta, \sigma)\in\Theta\times\Sigma\}$ is LAN 
			at any points $(\theta_0, \sigma_0)\in\Theta\times\Sigma$ for the rate matrix $\bar{\phi}_N(\theta_0, \sigma_0)$ in a weak sense, 
			that is, the log-likelihood ratio admits the following representation for any $u\in\mathbb{R}^{p+1}:$
			\begin{align*}
				\log{\frac{dP_{(\theta_0, \sigma_0)+\bar{\phi}_N(\theta_0, \sigma_0)u}^{(\delta)}}{dP_{\theta_0, \sigma_0}^{(\delta)}}}
				=\langle u, \bar{\zeta}_N(\theta_0, \sigma_0)\rangle - \frac{1}{2}\langle \mathcal{I}(\theta_0, \sigma_0)u, u\rangle + \bar{r}_N(\theta_0, \sigma_0),
			\end{align*}
			where
			\begin{align*}
				\bar{\zeta}_N(\theta_0, \sigma_0)\rightarrow\mathcal{N}(0, \mathcal{I}(\theta_0, \sigma_0)),\hspace{0.2cm} \bar{r}_N(\theta_0, \sigma_0)\rightarrow 0,
			\end{align*}
			in law under $P_{\theta_0, \sigma_0}^{(\delta)}$ and the asymptotic Fisher information matrix $\mathcal{I}(\theta, \sigma)$ is given by
			\begin{align*}
				\begin{pmatrix}
					\mathcal{F}_{p-1}(\theta)&a_{p-1}(\theta)&\frac{1}{\sigma}a_{p-1}(\theta) \\
					a_{p-1}(\theta)^{\mathrm{T}}& 2 &\frac{2}{\sigma} \\
					\frac{1}{\sigma}a_{p-1}(\theta)^{\mathrm{T}}& \frac{2}{\sigma} &\frac{2}{\sigma^2}
				\end{pmatrix}.
			\end{align*}
			In particular, the asymptotic Fisher information matrix $\mathcal{I}(\theta, \sigma)$ is singular 
			unless either the Hurst parameter $H$ or the diffusion parameter $\sigma$ is known. 
		\end{thm}
		
		Next, we show the extension of the result in Brouste and Fukasawa (2016). 
		Here, we introduce a certain class of rate matrices. 
		\begin{assump}\label{assump.rate}
			Consider a matrix 
			\begin{align*}
				\phi_N=\phi_N(\theta, \sigma)=
				\begin{bmatrix}
					\phi_{N, 1}&0_{p\times 1} \\
					0_{1\times p}&\phi_{N, 2}
				\end{bmatrix}, 
			\end{align*}
			where
			\begin{align*}
				\phi_{N, 1}=&\phi_{N, 1}(\theta, \sigma)=\mathrm{diag}(d_N^{(1)}, \cdots, d_N^{(p-1)})=
				\begin{pmatrix}
					d_N^{(1)}&&O \\
					&\ddots& \\
					O&&d_N^{(p-1)}
				\end{pmatrix}, \\
				\phi_{N, 2}=&\phi_{N, 2}(\theta, \sigma)=
				\begin{bmatrix}
					\alpha_N&\hat{\alpha}_N \\
					\beta_N&\hat{\beta}_N
				\end{bmatrix},
			\end{align*}
			with the following properties:
			\begin{enumerate}
				\item $|\phi_{N, 1}|=d_N^{(1)}\cdots d_N^{(p-1)}\neq 0$ and $|\phi_{N, 2}|=\alpha_N\hat{\beta}_N-\hat{\alpha}_N\beta_N\neq 0$.
				\item $\alpha_N\sqrt{N}\rightarrow\alpha$ for some $\alpha\in\mathbb{R}$.
				\item $\hat{\alpha}_N\sqrt{N}\rightarrow\hat{\alpha}$ for some $\hat{\alpha}\in\mathbb{R}$.
				\item For $j=1, \cdots, p-1$, $d_N^{(j)}\sqrt{N}\rightarrow d^{(j)}$ for some $d^{(j)}\in\mathbb{R}\backslash\{0\}$.
				\item $\gamma_N:=\alpha_N\sqrt{N}\log\delta+\beta_N\sqrt{N}\sigma^{-1}\rightarrow\gamma$ for some $\gamma\in\mathbb{R}$.
				\item $\hat{\gamma}_N:=\hat{\alpha}_N\sqrt{N}\log\delta+\hat{\beta}_N\sqrt{N}\sigma^{-1}\rightarrow\hat{\gamma}$ for some $\hat{\gamma}\in\mathbb{R}$.
				\item $d^{(1)}\cdots d^{(p-1)}\neq 0$ and $\alpha\hat{\gamma}-\hat{\alpha}\gamma\neq 0$.
			\end{enumerate}
		\end{assump}
		Then, we obtain the following LAN property.
		\begin{thm}\label{thm.LAN}
			Suppose Assumption \ref{model.assump.}. The family of measures $\{P_{\theta, \sigma}^{(\delta)}; (\theta, \sigma)\in\Theta\times\Sigma\}$ 
			is LAN at any points $(\theta_0, \sigma_0)\in\Theta\times\Sigma$ for the rate matrix $\phi_N(\theta_0, \sigma_0)$ satisfying Assumption \ref{assump.rate}, 
			that is, the log-likelihood ratio admits the following representation for any $u\in\mathbb{R}^{p+1}:$
			\begin{align*}
				\log{\frac{dP_{(\theta_0, \sigma_0)+\phi_N(\theta_0, \sigma_0)u}^{(\delta)}}{dP_{\theta_0, \sigma_0}^{(\delta)}}}=\langle u, \zeta_N(\theta_0, \sigma_0)\rangle - \frac{1}{2}\langle \mathcal{J}(\theta_0, \sigma_0)u, u\rangle + r_N(\theta_0, \sigma_0),
			\end{align*}
			where
			\begin{align*}
				\zeta_N(\theta_0, \sigma_0)\rightarrow\mathcal{N}(0, \mathcal{J}(\theta_0, \sigma_0)),\hspace{0.2cm} r_N(\theta_0, \sigma_0)\rightarrow 0,
			\end{align*}
			in law under $P_{\theta_0, \sigma_0}^{(\delta)}$ and the matrix $\mathcal{J}(\theta, \sigma)$ is given by
			\begin{align*}
				&
				\begin{pmatrix}
					D&0_{p-1\times 2} \\
					0_{2\times p-1}& E
				\end{pmatrix}
				\begin{pmatrix}
					\mathcal{F}_{p}(\theta)& a_p(\theta) \\
					a_p(\theta)^{\mathrm{T}}& 2
				\end{pmatrix}
				\begin{pmatrix}
					D&0_{p-1\times 2} \\
					0_{2\times p-1}& E
				\end{pmatrix}^{\mathrm{T}},
			\end{align*}
			where 
			\begin{align*}
				D=D(\theta_0, \sigma_0)=\mathrm{diag}(d^{(1)}, \cdots, d^{(p-1)}) , 
				\hspace{0.2cm}
				E=E(\theta_0, \sigma_0)=
				\begin{pmatrix}
					\alpha& \gamma \\
					\hat{\alpha}& \hat{\gamma}
				\end{pmatrix}.
			\end{align*}
			In particular, the matrix $\mathcal{J}(\theta, \sigma)$ is nondegenerate.
		\end{thm}
		
		\subsection{The efficient estimation rate for $(\psi, H)$}
		As the rate matrix, we can take 
		\begin{align*}
			\phi_{N, 2}=\frac{1}{\sqrt{N}}
			\begin{pmatrix}
				1& 0 \\
				-\sigma\log{\delta}& \sigma
			\end{pmatrix},
		\end{align*}
		which gives $\alpha=1$, $\hat{\alpha}=0$, $\gamma=0$ and $\hat{\gamma}=\sigma^{-1}$. 
		Therefore, 
		\begin{align*}
			E=
			\begin{pmatrix}
				1&0 \\
				0&1
			\end{pmatrix}.
		\end{align*}
		Then, $\mathcal{J}(\theta, \sigma)$ is given by
		\begin{align*}
			\mathcal{J}(\theta, \sigma)=
			\begin{bmatrix}
				\mathcal{F}_p(\theta)& a_p(\theta) \\
				a_p(\theta)^{\mathrm{T}}& 2
			\end{bmatrix}^{-1}.
		\end{align*}
		By Theorem \ref{thm.LAN}, the LAN property implies that 
		\begin{align*}
			\liminf_{\eta\to 0}\liminf_{\delta\to 0}\sup_{|(\theta, \sigma)-(\theta_0, \sigma_0)|<\eta}&E_{\theta, \sigma}\left[l\left(\phi_N^{-1}
			\begin{pmatrix}
				\hat{\theta}_N - \theta \\
				\hat{\sigma}_N - \sigma
			\end{pmatrix}
			\right)\right] \\
			&\geq E_{\theta_0, \sigma_0}\left[l\left(\mathcal{J}(\theta_0, \sigma_0)^{-1}N \right)\right]
		\end{align*}
		for any symmetric, nonnegative quasi-convex function $l$ with $\lim_{|z|\to\infty}e^{-\varepsilon|z|^2}l(z)=0$ for any $\varepsilon>0$, 
		where $N\sim\mathcal{N}(0, I_{p+1})$. Since 
		\begin{align*}
			\phi_N^{-1}=
			\begin{pmatrix}
				\frac{1}{\sqrt{N}}I_{p-1}&0_{p-1\times 2} \\
				0_{2\times p-1}& \phi_{N, 2}^{-1}
			\end{pmatrix}
			\quad\mbox{where}\quad\phi_{N, 2}^{-1}=\sqrt{N}
			\begin{pmatrix}
				1& 0 \\
				\log{\delta}& \frac{1}{\sigma}
			\end{pmatrix},
		\end{align*}
		we obtain the asymptotic lower bound of
		\begin{align*}
			\liminf_{\eta\to 0}\liminf_{\delta\to 0}\sup_{|(\theta, \sigma)-(\theta_0, \sigma_0)|<\eta}&E\left[
			\begin{pmatrix}
				\sqrt{N}(\hat{\psi}_N-\psi) \\
				\sqrt{N}(\hat{H}_N-H) \\
				\sqrt{N}\log\delta(\hat{H}_N-H_0)+\frac{\sqrt{N}}{\sigma}(\hat{\sigma}_N-\sigma)
			\end{pmatrix}^{\otimes 2}
			\right]
		\end{align*}
		by taking $l(x)=x^{\mathrm{T}}x, x\in\mathbb{R}^{p+1}$ as follows.
		\begin{align*} 
			\begin{bmatrix}
				\mathcal{F}_p(\theta_0)& a_p(\theta_0) \\
				a_p(\theta)^{\mathrm{T}}& 2
			\end{bmatrix}^{-1}
			=
			\begin{bmatrix}
				\mathcal{G}_p(\theta)^{-1}& -\frac{1}{2}\mathcal{G}_p(\theta)^{-1}a_p(\theta) \\
				-\frac{1}{2}a_p(\theta)^{\mathrm{T}}\mathcal{G}_p(\theta)^{-1}& \frac{1}{2}+\frac{1}{4}a_p(\theta)^{\mathrm{T}}\mathcal{G}_p(\theta)^{-1}a_p(\theta)
			\end{bmatrix}.
		\end{align*}
		Here, we use Lemma \ref{mat.eq.} in the above equality. 
		In particular, this means the efficient rate of estimation for $(\psi, H)$ is $\sqrt{N}I_p$ 
		when all parameters $(\psi, H, \sigma)$ are unknown. 
		Note that when $\sigma$ is known, the efficient rate for $H$ is $\sqrt{N}\log\delta$, 
		which follows from Theorem \ref{Kawai.extension}.
		
		\subsection{The efficient estimation rate for $\sigma$}
		As the rate matrix, we can take 
		\begin{align*}
			\phi_{N, 2}=\frac{1}{\sqrt{N}}
			\begin{pmatrix}
				\frac{1}{\log{\delta}}& 1 \\
				0& -\sigma\log{\delta}
			\end{pmatrix},
		\end{align*}
		which gives $\alpha=0$, $\hat{\alpha}=1$, $\gamma=1$ and $\hat{\gamma}=0$. 
		Therefore, 
		\begin{align*}
			E=
			\begin{pmatrix}
				0&1 \\
				1&0
			\end{pmatrix}.
		\end{align*}
		Then, $\mathcal{J}(\theta, \sigma)$ is given by
		\begin{align*}
			\mathcal{J}(\theta, \sigma)=
			\begin{pmatrix}
				I_{p-1}&0_{p-1\times 2} \\
				0_{2\times p-1}&E
			\end{pmatrix}
			\begin{bmatrix}
				\mathcal{F}_p(\theta)& a_p(\theta) \\
				a_p(\theta)^{\mathrm{T}}& 2
			\end{bmatrix}
			\begin{pmatrix}
				I_{p-1}&0_{p-1\times 2} \\
				0_{2\times p-1}&E
			\end{pmatrix}^{\mathrm{T}}.
		\end{align*}
		By Theorem \ref{thm.LAN}, the LAN property implies that 
		\begin{align*}
			\liminf_{\eta\to 0}\liminf_{\delta\to 0}\sup_{|(\theta, \sigma)-(\theta_0, \sigma_0)|<\eta}&E_{\theta, \sigma}\left[l\left(\phi_N^{-1}
			\begin{pmatrix}
				\hat{\theta}_N - \theta \\
				\hat{\sigma}_N - \sigma
			\end{pmatrix}
			\right)\right] \\
			&\geq E_{\theta_0, \sigma_0}\left[l\left(\mathcal{J}(\theta_0, \sigma_0)^{-1}N \right)\right]
		\end{align*}
		for any symmetric, nonnegative quasi-convex function $l$ with $\lim_{|z|\to\infty}e^{-\varepsilon|z|^2}l(z)=0$ for any $\varepsilon>0$, 
		where $N\sim\mathcal{N}(0, I_{p+1})$. Since 
		\begin{align*}
			\phi_N^{-1}=
			\begin{pmatrix}
				\frac{1}{\sqrt{N}}I_{p-1}&0_{p-1\times 2} \\
				0_{2\times p-1}& \phi_{N, 2}^{-1}
			\end{pmatrix}
			\quad\mbox{where}\quad\phi_{N, 2}^{-1}=-\frac{\sqrt{N}}{\sigma}
			\begin{pmatrix}
				-\sigma\log{\delta}& -1 \\
				0& \frac{1}{\log{\delta}}
			\end{pmatrix},
		\end{align*}
		we obtain the asymptotic lower bound of
		\begin{align*}
			\liminf_{\eta\to 0}\liminf_{\delta\to 0}\sup_{|(\theta, \sigma)-(\theta_0, \sigma_0)|<\eta}&E\left[
			\begin{pmatrix}
				\sqrt{N}(\hat{\psi}_N-\psi_0) \\
				-\frac{\sqrt{N}}{\sigma\log\delta}(\hat{\sigma}_N-\sigma_0) \\
				\sqrt{N}\log\delta(\hat{H}_N-H_0)+\frac{\sqrt{N}}{\sigma}(\hat{\sigma}_N-\sigma_0)
			\end{pmatrix}^{\otimes 2}
			\right]
		\end{align*}
		by taking $l(x)=x^{\mathrm{T}}x, x\in\mathbb{R}^{p+1}$ as follows.
		\begin{align*} 
			\begin{bmatrix}
				\mathcal{F}_p(\theta_0)& a_p(\theta_0) \\
				a_p(\theta)^{\mathrm{T}}& 2
			\end{bmatrix}^{-1}
			=
			\begin{bmatrix}
				\mathcal{G}_p(\theta)^{-1}& -\frac{1}{2}\mathcal{G}_p(\theta)^{-1}a_p(\theta) \\
				-\frac{1}{2}a_p(\theta)^{\mathrm{T}}\mathcal{G}_p(\theta)^{-1}& \frac{1}{2}+\frac{1}{4}a_p(\theta)^{\mathrm{T}}\mathcal{G}_p(\theta)^{-1}a_p(\theta)
			\end{bmatrix}.
		\end{align*}
		Here, we use Lemma \ref{mat.eq.} and the relation:
		\begin{align*}
			\phi_N^{-1}
			\begin{pmatrix}
				\hat{\theta}_N - \theta \\
				\hat{\sigma}_N - \sigma
			\end{pmatrix}
			=
			\begin{pmatrix}
				I_{p-1}&0_{p-1\times 2} \\
				0_{2\times p-1}&E
			\end{pmatrix}
			\begin{pmatrix}
				\sqrt{N}(\hat{\psi}_N-\psi_0) \\
				-\frac{\sqrt{N}}{\sigma\log\delta}(\hat{\sigma}_N-\sigma_0) \\
				\sqrt{N}\log\delta(\hat{H}_N-H_0)+\frac{\sqrt{N}}{\sigma_0}(\hat{\sigma}_N-\sigma_0)
			\end{pmatrix}.
		\end{align*}
		In particular, this means the efficient rate of estimation for $\sigma$ is $\sqrt{N}/|\log\delta|$ 
		when all parameters $(\theta, \sigma)$ are unknown. 
		Note that when $H$ is known, the efficient rate for $\sigma$ is $\sqrt{N}$, 
		which follows from Theorem \ref{Kawai.extension}.

\begin{thebibliography}{99}
			\bibitem{BF} Brouste, A. and Fukasawa, M. (2016): Local asymptotic normality property for fractional Gaussian noise under high frequency observations, arXiv:1610.03694v1.
			\bibitem{BI} Brouste, A. and Iacus, S. M. (2013): Parameter estimation for the discretely observed fractional Ornstein-Uhlenbeck process and the Yuima R package, 
			Comput Stat, 28:1529-1547.
			\bibitem{CGLL} Cohen, S., Gamboa, F., Lacaux, C. and Loubes, J. M. (2013): LAN property for some fractional type Brownian motion, 
			ALEA : Latin American Journal of Probability and Mathematical Statistics, 10(1), 91-106.
			\bibitem{Coeurjolly} Coeurjolly, J. F. (2001): Estimating the Parameters of a Fractional Brownian Motion by Discrete Variations of its Sample Paths, 
			Statistical Inference for Stochasitic Processes, 4, 199-227.
			\bibitem{Dahlhaus} Dahlhaus, R. (1989): Efficient parameter estimation for self-similar processes, The Annals of Statistics, Vol. 17, No. 4, 1749-1766.
			\bibitem{Dahlhaus.2006} Dahlhaus, R. (2006): Correction note: Efficient parameter estimation for self-similar processes, The Annals of Statistics, Vol. 34, No. 2, 1045-1047.
			\bibitem{FT1986} Fox, R. and Taqqu, M. S. (1986): Large-Sample Properties of Parameter Estimates for Strongly Dependent Stationary Gaussian Time Series, The Annals of Statistics, Vol. 14, No. 2, 517-532.
			\bibitem{IH} Ibragimov, I. A. and Has'minski, R. Z. (1981): Statisticial estimation: asymptotic theory, Springer-Verlag.
			\bibitem{Kawai} Kawai, R. (2013): Fisher Information for Fractional Brownian Motion Under High-Frequency Discrete Sampling, Communications in Statistics - Theory and Methods, Vol. 42, 1628-1636.
			\bibitem{Kolmogorov} Kolmogorov, A. N. (1940): Wienersche Spiralen und einige andere interessante Kurven im Hilbertschen Raum (German), 
			C. R. (Doklady) Acad. Sci. URSS (N.S.) 26, 115-118. 
			\bibitem{LRR} Lieberman , O. , Rosemarin , R. and Rousseau , J. (2009): Asymptotic Theory for Maximum Likelihood Estimation in Stationary Fractional Gaussian Processes, Under Short-, Long- and Intermediate Memory, CiteSeerx.
			\bibitem{LRR2} Lieberman , O. , Rosemarin , R. and Rousseau , J. (2011): Asymptotic Theory for Maximum Likelihood Estimation of memory parameter in Stationary Gaussian Processes, HAL, hal-00641474.
			\bibitem{LRR3} Lieberman , O. , Rosemarin , R. and Rousseau , J. (2012): Asymptotic Theory for Maximum Likelihood Estimation of memory parameter in Stationary Gaussian Processes, Econometric Theory, Vol. 28, 457-470.
			\bibitem{MV} Mandelbrot, B. and Van Ness, J. W. (1968): Fractional Brownian motions, fractional noises and applications, 
			SIAM Review, 10, 422–437. 
			\bibitem{Nourdin} Nourdin, I. (2012): Selected Aspects of Fractional Brownian Motion, Springer.
			\bibitem{Rosemarin} Rosemarin, R. (2008): Maximum Likelihood Estimation in Fractional Gaussian Stationary and Invertible Processes, Master thesis from Tel Aviv university in Israel.
		\end{thebibliography}
	\end{document}